\documentclass{amsart}
\usepackage{tikz-cd}

\theoremstyle{plain}
\newtheorem{thm}{Theorem}[section]

\newtheorem{lem}[thm]{Lemma}

\newcommand{\Z}{\mathbb{Z}}
\newcommand{\N}{\mathbb{N}}
\newcommand{\C}{\mathcal C}

\DeclareMathOperator{\lmod}{-mod}
\DeclareMathOperator{\lMod}{-Mod}
\DeclareMathOperator{\tr}{tr}
\DeclareMathOperator{\Hom}{Hom}
\DeclareMathOperator{\Ker}{Ker}

\DeclareMathOperator{\Ima}{Im}

\DeclareMathOperator{\End}{End}
\DeclareMathOperator{\Ext}{Ext}
\DeclareMathOperator{\rank}{rank}
\DeclareMathOperator{\dimv}{\underline{\dim}}

\begin{document}
\title{A new approach to simple modules for preprojective algebras}

\author{William Crawley-Boevey}
\address{Fakult\"at f\"ur Mathematik, Universit\"at Bielefeld, 33501 Bielefeld, Germany}
\email{wcrawley@math.uni-bielefeld.de}

\author{Andrew Hubery}
\email{ahubery@math.uni-bielefeld.de}



\thanks{Both authors have been supported by the Alexander von Humboldt Foundation 
in the framework of an Alexander von Humboldt Professorship 
endowed by the German Federal Ministry of Education and Research.}

\begin{abstract}
The work of the first author on the moment map for representations of quivers
included a classification of the possible dimension vectors of simple modules for deformed preprojective
algebras. That classification was later used to solve an additive analogue of the Deligne-Simpson problem.
The last step in the proof of the classification involved some general position arguments; 
here we give a new approach which avoids such arguments.
\end{abstract}
\maketitle

\section{Introduction}
Let $K$ be an algebraically closed field and let $Q$ be a quiver with vertex set~$I$. Recall that a left module $X$ for the path algebra $K Q$ corresponds to a representation of $Q$, so is determined by vector spaces $X_i$ for each vertex $i\in I$ and linear maps $X_a:X_{t(a)}\to X_{h(a)}$ for each arrow $a\in Q$, where $h(a)$ and $t(a)$ are the head and tail vertices of~$a$. If finite-dimensional, the dimension vector is $\dimv X := (\dim X_i)\in\N^I$.

The \emph{double} of $Q$ is the quiver $\bar Q$ obtained by adjoining an arrow $a^\ast\colon j\to i$ for each arrow $a\colon i\to j$ of $Q$,
and given $\lambda\in K^I$, the \emph{deformed preprojective algebra} $\Pi^\lambda Q$ is the quotient of $K\bar Q$ 
by the ideal generated by the relation
\[ 
c := \sum_{a\in Q}[a,a^\ast] - \sum_i\lambda_i e_i. 
\]
Here $[a,a^\ast]:=aa^\ast-a^\ast a$ and $e_i$ is the idempotent corresponding to vertex $i$. Equivalently we can take the ideal generated by the elements $c_i=e_i c=c e_i$, one for each vertex (see~\cite{CBH}). Thus a left module for $\Pi^\lambda Q$ corresponds to a representation $X$ of $\bar Q$ satisfying $X_{c,i} = 0$ for all vertices $i$, where we define
\[
X_{c,i} := \sum_{\substack{a\in Q \\h(a)=i}} X_a X_{a^*} - \sum_{\substack{a\in Q \\t(a)=i}} X_{a^*} X_a - \lambda_i 1_{X_i}\in\End_K(X_i).
\]
As usual, by considering the sum of the traces of $X_{c,i}$, one sees that if $X$ is a $\Pi^\lambda Q$-module, then $\lambda\cdot\dimv X := \sum_i\lambda_i\dim X_i$ must be zero.

The set of dimension vectors $\Sigma_\lambda$ of the finite dimensional simple $\Pi^\lambda Q$-modules was determined in \cite{CBmm}, 
and described purely combinatorially in terms of $Q$ and $\lambda$. 
This was later used in \cite{CBadsp} to solve an additive analogue of the Deligne-Simpson Problem.
The proof that the dimension vectors of simple modules must lie in (the combinatorially defined set) $\Sigma_\lambda$ reduces to showing that there are no simple modules of a certain dimension vector in three specific cases. The argument in \cite{CBmm}, and in particular the proof of Theorem 9.1 of that paper, used general position arguments. In this paper we offer an alternative approach.

Let $v$ be a vertex of $Q$. By a \emph{nearly representation} of $\Pi^\lambda(Q)$ (with respect to $v$) we mean a representation $X$ of $\bar Q$ with $\lambda\cdot\dimv X=0$ such that $X_{c,i}=0$ for all $i\neq v$, whereas $\rank(X_{c,v})\le 1$.

Let $Q$ be an affine (or extended Dynkin) quiver, $\delta$ the minimal positive imaginary root and $v$ an extending vertex (meaning that $\delta_v=1$, so deleting $v$ gives the corresponding Dynkin quiver). Let $\lambda\in K^I$. Our main theorem is as follows.

\begin{thm}
\label{t:main}
If $\lambda\cdot\delta=0$ and $X$ is a nearly representation of $\Pi^\lambda Q$ with $\dimv X = m \delta$ and $m > 1$, then $X$ is not simple. 
\end{thm}

The theorem is proved in section~\ref{s:proof}. In section~\ref{s:theotherbit} we explain how it fits with \cite{CBmm}.

\section{Almost commuting matrices}
The case of Theorem~\ref{t:main} when $Q$ is a loop, the ``Jordan quiver'', needs to be treated separately, as the base for an induction. It takes the following form.

\begin{lem} 
\label{l:loopcase}
If $a,b$ are endomorphisms of $K^m$ with $m>1$, whose commutator $c=ab-ba$ has rank at most one, then $a$ and $b$ have a non-trivial proper common invariant subspace $U$. 
\end{lem}

This is proved in \cite[Lemma 9.7]{CBmm}, and in fact has a longer history, see \cite{Laffey,Guralnick,CLR}. We offer the following mild reformulation of the proof in \cite{CBmm}. Replacing $a$ by $a-\alpha$ for an eigenvalue $\alpha\in K$, we may assume that $a$ is not invertible. If $a=0$, then we just take an invariant subspace of $b$.  Otherwise, observe that $c(\Ker(a))\subset\Ima(a)$, so either $\Ima(c)\subset\Ima(a)$ and we can take $U=\Ima(a)$, or else $\Ker(a)\subset\Ker(c)$ and we can take $U=\Ker(a)$.

\section{Nearly representations and a quiver $Q_\infty$}

We first note the following fact.

\begin{lem}\label{lem:case-II}
If $X$ is a representation of $\bar Q$ and $\dim X_v\le 1$, then $X$ is a nearly representation of $\Pi^\lambda Q$ if and only if it is an actual $\Pi^\lambda Q$-module.
\end{lem}

\begin{proof}
By definition, if $X$ is a nearly representation, then $\lambda\cdot\dimv X=0$. This ensures that $X_{c,v}$ has trace zero, and since it is an endomorphism of a vector space of dimension $\le 1$, it is zero.
\end{proof}

Let  $Q_\infty$ be the quiver obtained from $Q$ by adjoining a new vertex $\infty$, so its vertex set is $I_\infty = I \cup \{\infty\}$, and a new arrow $a\colon\infty\to v$. We extend $\lambda\in K^I$ by zero to consider it as as an element of $K^{I_\infty}$, that is, we define $\lambda_\infty = 0$. In this section we explain the relationship between nearly representations of $\Pi^\lambda Q$ with respect to $v$, and representations of $\Pi^\lambda Q_\infty$. We write $S(\infty)$ for the simple $\Pi^\lambda Q_\infty$-module which is 1-dimensional at vertex $\infty$ and zero elsewhere. We write $\C$ for the category of $\Pi^\lambda Q_\infty$-modules $X$ with $\Hom(X,S(\infty))=\Hom(S(\infty),X)=0$. There is a forgetful functor from the category of representations of $\bar Q_\infty$, the double of $Q_\infty$, to the category of representations of~$\bar Q$.

\begin{lem}
\label{l:nearlyinfty}
The forgetful functor induces an equivalence from the category of $X\in \C$ with $\dim X_\infty\le 1$ to the category of nearly representations of $\Pi^\lambda Q$. 
\end{lem}

\begin{proof}
Let $A$ be the algebra defined like the deformed preprojective algebra $\Pi^\lambda Q_\infty$ but omitting the relation at vertex $\infty$, and consider the idempotent $f = \sum_{i\in I} e_i = 1-e_\infty\in A$. Clearly $A/(f)\cong K$. Now let $B$ be the algebra defined like $\Pi^\lambda Q$ but omitting the relation at vertex $v$. Then the natural map $B\to fAf$ is an isomorphism. For, it is is surjective since any path in $\overline{Q_\infty}$ with start and end vertices in $I$ only involves $a$ and $a^*$ via the product $aa^*$, and using the relation at $v$ this can be written as a linear combination of paths in $\bar Q$. Now it is an isomorphism since any $B$-module $X$ can be turned into an $A$-module $\ell(X)$ by defining $X_\infty:=X_v$, $X_a := -X_{c,v}$ and $X_{a^*}:= 1_{X_v}$, where $X_{c,v}$ is taken with respect to the quiver $Q$.

We get a recollement
\[ \begin{tikzcd}[column sep=50pt]
A/(f)\lMod \arrow[r, "i" description] &
A\lMod \arrow[r, "f" description] \arrow[l, yshift=8pt, swap, "q"] \arrow[l, yshift=-8pt, "p"] &
B\lMod \arrow[l, yshift=8pt, swap, "\ell"] \arrow[l, yshift=-8pt, "r"]
\end{tikzcd} \]
where $f$ denotes the functor $X\leadsto f X$. The left adjoint $\ell$ is as above, and the right adjoint $r$ is similar, with $X_a := 1_{X_v}$ and $X_{a^\ast} := -X_{c,v}$. There is an intermediate extension functor $\gamma:B\lMod\to A\lMod$ where $\gamma(X)$ is the image of the natural map $\ell(X)\to r(X)$, so $\gamma(X)_\infty \cong \Ima X_{c,v}$. An $A$-module is said to be bistable if it has no homomorphism to or from a non-zero $A/(f)$-module, or equivalently in this case, to or from the simple module $S(\infty)$. By standard properties of recollements, see for example \cite[Lemma 2.2]{CBS} and the remark following \cite[Definition 2.5]{CBS}, the functor $f$ induces an equivalence from the category of bistable $A$-modules to $B\lMod$, with inverse $\gamma$. Finally, by a trace argument as in Lemma~\ref{lem:case-II}, a representation of $A$ whose dimension at vertex $\infty$ is at most one is a $\Pi^\lambda Q_\infty$-module if and only if $\lambda\cdot\dimv X=0$.
\end{proof}

There is the following consequence (see for example \cite[Lemma 2.1(6)]{CBS}).

\begin{lem}
The forgetful functor induces a bijection between isomorphism classes of simple $\Pi^\lambda Q_\infty$-modules $X$ with $\dim X_\infty\le 1$ other than $S(\infty)$ and isomorphism classes of simple nearly representations of $\Pi^\lambda Q$.
\end{lem}

\section{Treating representations of $\bar Q$ as pairs $(X,\xi)$}
Given representations $M$ and $N$ of $Q$, or more generally $I$-graded vector spaces, we define
\[
h(N,M) = \bigoplus_{i\in I} \Hom_K(N_i,M_i),
\quad
r(N,M) =  \bigoplus_{a\in Q}\Hom_K(N_{h(a)},M_{t(a)}).
\]
If we have finite direct sum decompositions $N = \bigoplus_\nu N^\nu$ and $M = \bigoplus_\mu M^\mu$, then 
\[
h(N,M) = \bigoplus_{\nu,\mu} h(N^\nu,M^\mu),
\quad
r(N,M) = \bigoplus_{\nu,\mu} r(N^\nu,M^\mu),
\]
so elements of $h(N,M)$ and $r(N,M)$ can be considered as block matrices whose entries are in $h(N^\nu,M^\mu)$ and $r(N^\nu,M^\mu)$.

The following lemma is standard (the case $M=N$ is in \cite[\S 3]{CBmm}).

\begin{lem}
\label{lem:four-term}
Let $Q$ be a quiver, and $M$ and $N$ finite-dimensional $KQ$-modules. Then the map
\[ 
\Phi_{MN} : r(N,M) \to h(N,M), \quad  \theta \mapsto \sum_a(M_a\theta_a-\theta_aN_a), 
\]
has kernel $D\Ext^1_{KQ}(M,N)$ and cokernel $D\Hom_{KQ}(M,N)$, where $D=\Hom_K(-,K)$ is the usual vector space duality. In particular
\[ 
\sum_{i\in I} \tr \big(\Phi_{MN}(\theta)_i f_i\big) = 0 
\]
for all $f\in\Hom_{KQ}(M,N)$ and $\theta\in r(N,M)$.
\end{lem}

\begin{proof}
We start from the standard resolution of $M$,
\[ 
0 \to \bigoplus_{a\colon i\to j}KQe_i\otimes_KM_j \to \bigoplus_iKQe_i\otimes_KM_i \to M \to 0, 
\]
apply $\Hom_{KQ}(-,N)$, and then dualise. We then use the trace pairing to identify $D\Hom_K(U,V)$ with $\Hom_K(V,U)$ for all vector spaces $U$ and $V$.
\end{proof}

It will be convenient for us to regard representations of $\bar Q$ as pairs $(X,\xi)$ where $X$ is a representation of $Q$ and $\xi\in r(X,X)$ gives the action of the extra arrows $a^\ast\in\bar Q$. Given $\lambda\in K^I$, we write $\lambda 1_X$ for the element of $h(X,X)$ whose $i$th component is $\lambda_i 1_{X_i}$. Clearly the $\Pi^\lambda Q$-modules are the pairs $(X,\xi)$ such that $\Phi_{XX}(\xi)=\lambda 1_X$. Similarly, $(X,\xi)$ is a nearly representation of $\Pi^\lambda Q$ if $\Phi_{XX}(\xi) - \lambda 1_X$ has $v$-th component of rank at most one, all other components are zero, and $\lambda\cdot\dimv X = 0$.

\section{Affine quivers}
Henceforth $Q$ is affine with minimal positive imaginary root $\delta$ and $v$ is an extending vertex. We write $\langle-,-\rangle$ for the Euler form of $Q$ on $\Z^I$.
The reason to consider representations of $\bar Q$ as pairs $(X,\xi)$ is that we can then bring in the representation theory of $Q$. Let us recall it in this case (see \cite[Theorem 3.6(5)]{Rin}). 

If $Q$ is has no oriented cycles, then $K Q$ is finite dimensional, and every indecomposable $KQ$-module $X$ is either preprojective, regular, or preinjective, depending on whether the defect $\langle\delta,\dimv X\rangle$ is negative, zero, or positive. Moreover, the regular modules form a thick abelian subcategory, which further decomposes into a direct sum of uniserial categories, called tubes, and indexed by the projective line. Finally, each tube has finite period (that is, finitely many simple objects), and all but at most three of these tubes are homogeneous, meaning they have period 1, or equivalently a unique simple object.

Since $Q$ is affine, the other possibility is that $Q$ is an oriented cycle. Then every finite dimensional indecomposable module lies in a tube, and at most one tube is not homogeneous (see \cite[Theorem 3.6(6)]{Rin}). In this case we say that all finite dimensional modules are regular.

\begin{lem}
\label{l:homexist}
Suppose $Q$ has no oriented cycles. Let $P$ and $I$ be two $KQ$-modules with $P$ preprojective and $I$ preinjective. Then given any non-zero $x\in P_v$ and $y\in I_v$, there exists $f\in\Hom_{KQ}(P,I)$ with $f(x)=y$.
\end{lem}

\begin{proof}
Let $P(v) = KQ e_v$, a projective module of defect $-1$. If $M$ is preprojective, then any non-zero homomorphism $g\colon P(v)\to M$ is injective, and the cokernel has no preinjective summand. Namely, the image of $g$ is non-zero preprojective, so has negative defect, hence the kernel of $g$ is preprojective with defect $\geq0$, so is zero. Next, if we have a preinjective summand of the cokernel, then the pull-back along this inclusion yields a submodule $N\leq M$ of defect $\geq0$, a contradiction since $N$ must be preprojective.

Now, given $0\neq x\in P_v$, we have the associated short exact sequence
\[ 
0 \to P(v) \xrightarrow{g} P \to C \to 0, \quad g(v)=x, 
\]
where $C$ has no preinjective summand. Applying $\Hom_{KQ}(-,I)$ and using that $\Ext^1_{KQ}(C,I)=0$ yields an epimorphism
\[ 
\Hom_{KQ}(P,I) \to \Hom_{KQ}(P(v),I)\cong I_v. 
\]
In particular, given $y\in I_v$, there exists some $f\colon P\to I$ such that $f(x)=y$.
\end{proof}

\section{Proof in the non-regular case}
We prove Theorem~\ref{t:main} in the special case of representations $(X,\xi)$ of $\bar Q$, where the representation $X$ of $Q$ is not regular.

\begin{lem}
\label{l:nonreg}
Suppose that $\lambda\cdot\delta=0$. If $(X,\xi)$ is a nearly representation of $\Pi^\lambda Q$ with $\dimv X = m \delta$, $m > 1$ and $X$ is not regular, then $(X,\xi)$ is not simple. 
\end{lem}

\begin{proof}
Write $X=P\oplus R\oplus I$ with $P$ preprojective, $R$ regular and $I$ preinjective. Since $X$ has zero defect but is not regular, we know that both $P$ and $I$ are non-zero. Now $\xi \in r(X,X)$ so it can be considered as a block matrix as in section 4
\[ 
\xi = \begin{pmatrix}\xi_{PP}&\xi_{PR}&\xi_{PI}\\\xi_{RP}&\xi_{RR}&\xi_{RI}\\\xi_{IP}&\xi_{IR}&\xi_{II}\end{pmatrix}. 
\]
Since the decomposition of $X$ is into $KQ$-module direct summands, for $a\in Q$ the linear maps defining $X$ as a representation take block form
\[
X_a =  \begin{pmatrix} P_a & 0&0 \\ 0& R_a & 0\\ 0&0&I_a\end{pmatrix}. 
\]
It follows that $\Phi_{XX}(\xi) - \lambda 1_X$ takes block form
\[
\Phi_{XX}(\xi) - \lambda 1_X = \begin{pmatrix}\Phi_{PP}(\xi_{PP}) - \lambda 1_P&\Phi_{PR}(\xi_{PR})&\Phi_{PI}(\xi_{PI}) \\ \Phi_{RP}(\xi_{RP}) &\Phi_{RR}(\xi_{RR}) - \lambda 1_R&\Phi_{RI}(\xi_{RI}) \\ \Phi_{IP}(\xi_{IP})&\Phi_{IR}(\xi_{IR})&\Phi_{II}(\xi_{II}) - \lambda 1_I \end{pmatrix}. 
\]

Consider the block $\theta:=\xi_{PI}\in r(I,P)$. Since $\Ext^1_{KQ}(P,I)=0$, the four-term sequence of Lemma~\ref{lem:four-term} takes the form
\[ 
0 \to r(I,P) \xrightarrow{\Phi_{PI}} h(I,P) \to D\Hom_{KQ}(P,I) \to 0. 
\]
In particular $\Phi_{PI}$ is injective. Since $(X,\xi)$ is a nearly representation, it follows that $\Phi_{PI}(\theta)_v$ has rank at most one, whereas all other components are zero. Suppose $\Phi_{PI}(\theta)_v(y)=x\neq0$. By Lemma~\ref{l:homexist} there exists $f\colon P\to I$ with $f(x)=y$, and so
\[ 
1 = \tr (\Phi_{PI}(\theta)_v f_v) = \sum_i \tr (\Phi_{PI}(\theta)_i f_i) = 0, 
\]
a contradiction. Thus $\Phi_{PI}(\theta)=0$, and hence $\theta=0$.

We deduce that $\Phi_{XX}(\xi)-\lambda 1_X$ at position $PI$ is zero. As the $v$ component has rank at most one and all other components are zero, we deduce that one of the blocks $PR$ or $RI$ must also be zero.

Assume $\Phi_{PR}(\xi_{PR})=0$. Then $\Ext^1_{KQ}(P,R)=0$, so $\Phi_{PR}$ is injective and $\xi_{PR}=0$. Thus $R\oplus I$ is a $K\bar Q$-submodule of $(X,\xi)$. The argument when $\Phi_{RI}(\xi_{RI})=0$ is entirely analogous, in which case $I$ is a $K\bar Q$-submodule of $(X,\xi)$.
\end{proof}

\section{Universal localization and perpendicular categories}
\label{s:univloc}
If $\Lambda$ is a hereditary $K$-algebra and $T$ is a finitely presented $\Lambda$-module, then there is a universal localization $\Lambda\to \Lambda_T$ such that the restriction functor $\Lambda_T\lMod\to \Lambda\lMod$ is fully faithful and has essential image the perpendicular category
\[ 
T^\perp := \{ X \in \Lambda\lMod : \Hom_\Lambda(T,X) = 0 = \Ext^1_\Lambda(T,X) \}.
\]
We apply this to $\Lambda = K Q$ where $Q$ is an affine quiver without oriented cycles.

\begin{lem}
\label{l:perpT}
Given a tube $\mathcal T$ tube for $K Q \lmod$, there is a regular module $T$ with $\mathcal T \subseteq T^\perp$ and $KQ_T$ Morita equivalent to $K Q'$, where $Q'$ is an oriented cycle. Moreover the number of vertices of $Q'$ is equal to the period of the tube $\mathcal T$, which is strictly smaller than the number of vertices of $Q$. Also, every finite-dimensional module in $T^\perp$ is regular.
\end{lem}

\begin{proof}
This is analogous to \cite[Lemma 11.1]{CBconze}. Take $T = U\oplus S$ where $U$ is the direct sum of all but one regular simple in each tube other than $\mathcal T$,
and $S$ is one more regular simple not in $\mathcal T$. In particular $T$ involves all the regular simple modules in some tube, which ensures that $T^\perp$ only contains regular modules. According to \cite[\S4]{CBrmtha}, the universal localization with respect to $U$ is a finite dimensional tame hereditary algebra whose regular modules consist of a tube corresponding to $\mathcal T$ and with all other tubes homogeneous. Thus $KQ_U$ is Morita equivalent to the path algebra of a quiver $Q''$ of shape
\[ 
\begin{tikzcd}[row sep=2pt]
0 \arrow[dr, swap, "a_1"] \arrow[rrrr, bend left=20, "a_0"] &&&& n\\
& 1 \arrow[r, swap, "a_2"] & 2 \arrow[r, swap, "a_3"] & \cdots \arrow[ur, swap, "a_n"]
\end{tikzcd} 
\]
Now $KQ_T$ can also be thought of as the universal localization of $KQ_U$ with respect to a regular simple module $S'$ in a homogeneous tube. Now $S'$ corresponds to a representation $S''$ of $Q''$ of the form 
\[ 
\begin{tikzcd}[row sep=2pt]
K \arrow[dr, swap, "1"] \arrow[rrrr, bend left=20, "\zeta"] &&&& K\\
& K \arrow[r, swap, "1"] & K \arrow[r, swap, "1"] & \cdots \arrow[ur, swap, "1"]
\end{tikzcd} 
\]
for some $\zeta\in K$. Changing the generators of $KQ''$ we may assume that $\zeta=0$. Then $S''$ has projective resolution
\[ 
\begin{tikzcd}
0 \arrow[r] & P''(e_n) \arrow[r, "\cdot a_0"] & P''(e_0) \arrow[r] & S'' \arrow[r] & 0,
\end{tikzcd} 
\]
so modules for the universal localization $KQ''_{S''}$ correspond to representations of $Q''$ in which the arrow $a_0$ is an isomorphism. It follows that $KQ''_{S''}$, and hence also $KQ_T$, is Morita equivalent to the path algebra of the oriented cycle quiver $Q'$ obtained by shrinking $a_0$ to identify vertices $0$ and $n$.
\end{proof}

Morita equivalence followed by restriction gives an equivalence $\iota: K Q'\lMod \to T^\perp$. Moreover, this induces an equivalence between the category of finite dimensional (respectively nilpotent) $KQ'$-modules and the finite dimensional modules in $T^\perp$ (respectively $\mathcal T$).

The functor $\iota$ also has a left adjoint $t:KQ\lMod\to KQ'\lMod$, given by the tensor product functor $KQ_T \otimes_{KQ} -$ followed by Morita equivalence. Note that this sends projective $KQ$-modules to projective $KQ'$-modules. Moreover, every finitely generated projective $KQ'$-module is isomorphic to a direct sum of standard indecomposable projectives, which are those of the form $P'(j):=KQ'e_j$, indexed by the vertices of $Q'$. We can therefore compute
\[ \dim\Hom_{KQ'}\big(t(P),S(j)\big) = \dim\Hom_{KQ}\big(P,\iota(S(j))\big) = \big\langle\dimv P,\dimv\iota(S(j))\big\rangle. \]

\begin{lem}\label{l:mijprop}
Write $\delta'$ for the minimal imaginary root for $Q'$ (which has all components equal to one).

(a) We have $\lambda'\cdot\delta' = \lambda\cdot\delta$, where $\lambda'_j := \sum_i\lambda_i\big\langle\dimv P(i),\dimv \iota(S(j))\big\rangle$.

(b) $t(P(v))\cong P'(v')$ for some (necessarily extending) vertex $v'$ of $Q'$.

(c) If $Y$ is a $KQ'$-module, then $\dimv Y=m\delta'$ if and only if $\dimv\iota(Y)=m\delta$.
\end{lem}

\begin{proof}
We use that the dimension vectors of the regular simple modules in any tube sum to the minimal imaginary positive root. Thus
\[ \sum_j\dimv S(j) = \delta' \quad\textrm{and} \quad \sum_j\dimv\iota(S(j)) = \delta. \]

(a) We have $\lambda\cdot\delta = \sum_{i,j}\lambda_i\langle\dimv P(i),\dimv\iota S(j)\rangle = \lambda'\cdot\delta'$.

(b) This follows from $\sum_j \big\langle\dimv P(v),\dimv \iota(S(j))\big\rangle = 1$.

(c) We know that $\dimv Y$ is a multiple of $\delta'$ if and only if
\[ \dim\Hom_{KQ'}(Y,Y) = \dim\Ext^1_{KQ'}(Y,Y), \]
which since $\iota$ is an exact embedding is if and only if $\big\langle\dimv\iota(Y),\dimv\iota(Y)\big\rangle=0$, if and only if $\dimv\iota(Y)$ is a multiple of $\delta$. In this case we can compute the specific multiple of $\delta'$ or $\delta$ as
\[ m = \dim\Hom_{KQ'}\big(P'(v),Y\big) = \dim\Hom_{KQ}\big(P(v),\iota(Y)\big). \qedhere \]
\end{proof}

Now \cite[Corollary 9.6]{CBconze} gives the following:

\begin{lem}
There is an equivalence of categories between 
\begin{enumerate}
\item $\Pi^{\lambda'} Q'$-modules $(Y,\eta)$ (respectively $Y$ nilpotent), and
\item $\Pi^\lambda Q$-modules $(X,\xi)$ with $X\in T^\perp$ (respectively $X\in\mathcal T$).
\end{enumerate}
It sends $(Y,\eta)$ to $(\iota(Y),\xi)$ for some $\xi$.
\end{lem}

Our main result in this section is the analogue for nearly representations, with respect to the vertices $v$ and $v'$ as above.

\begin{lem}
\label{l:nearlyperp}
There is an equivalence of categories between 
\begin{enumerate}
\item Nearly representations $(Y,\eta)$ of $\Pi^{\lambda'} Q'$ (respectively $Y$ nilpotent), and
\item Nearly representations $(X,\xi)$ of $\Pi^\lambda Q$ with $X\in T^\perp$ (respectively $X\in\mathcal T$).
\end{enumerate}
It sends $(Y,\eta)$ to $(\iota(Y),\xi)$ for some $\xi$.
\end{lem}

\begin{proof}
The path algebra $K Q_\infty$ can be considered as a one-point extension of $K Q$,
\[
K Q_\infty \cong \begin{pmatrix}KQ & P(v) \\ 0 & K \end{pmatrix}.
\]
By \cite[Lemma 9.6]{CBmm} we obtain a pseudoflat epimorphism (in fact a universal localisation) $K Q_\infty \to B$ where
\[
B = \begin{pmatrix}KQ_T & KQ_T\otimes_{KQ} P(v) \\ 0 & K \end{pmatrix},
\]
and using that $t(P(v))\cong P'(v')$ we see that $B$ is Morita equivalent to $KQ'_\infty$. We obtain a functor $\iota\colon KQ'_\infty\lMod\to KQ_\infty\lMod$ lifting $\iota\colon KQ'\lMod\to KQ\lMod$ and sending $S(\infty)$ to $S(\infty)$.

By \cite[Theorems 0.2, 0.4, 0.7]{CBconze} we can lift this to a map $\Pi^\lambda Q_\infty\to\Pi^\mu(B)$, with $\Pi^\mu(B)$ Morita equivalent to $\Pi^{\lambda'}Q'_\infty$, and such that restriction of scalars gives an equivalence between
\begin{enumerate}
\item $\Pi^{\lambda'} Q'_\infty$-modules, and
\item $\Pi^\lambda Q_\infty$-modules $(X,\xi)$ such that the restriction of $X$ to $Q$ is in $T^\perp$.
\end{enumerate}
This sends $(Y,\eta)$ to $(\iota(Y),\xi)$ for some $\xi$. Now use Lemma~\ref{l:nearlyinfty}, noting that $\dim Y_\infty=\dim \iota(Y)_\infty$.
\end{proof}

\section{Proof of the main theorem}
\label{s:proof}
Let $Q$ be an affine quiver, $\delta$ the minimal positive imaginary root and $v$ an extending vertex. In this section we prove the following result. Combined with Lemma~\ref{l:nonreg} it gives Theorem~\ref{t:main}.

\begin{thm}\label{thm:main}
Suppose that $\lambda\cdot\delta=0$. If $(X,\xi)$ is a nearly representation of $\Pi^\lambda Q$ with $\dimv X = m\delta$, $m>1$ and $X$ is regular, then $(X,\xi)$ has a proper non-trivial submodule $(X',\xi')$ with $X'$ regular.
\end{thm}

\begin{proof}
We proceed by induction on $m$ and the number of vertices in $Q$.

\medskip
\noindent
\textit{Case 1. (Multiple tubes).}
Suppose there is a non-trivial decomposition $X=U\oplus V$ with the indecomposable summands of $U$ and $V$ lying in disjoint sets of tubes. Note that $\langle\dimv V,\dimv V\rangle = \langle m\delta-\dimv U,\dimv V\rangle = 0$ since $\Hom(U,V) = \Ext^1(U,V)=0$, so $\dimv =m'\delta$ for some $m'$. Writing $\xi$ in block form
\[
\xi = \begin{pmatrix}\xi_{UU}&\xi_{UV}\\\xi_{VU}&\xi_{VV}\end{pmatrix}
\]
according to the decomposition and setting $\theta:=\xi_{VV}$, we see that $(V,\theta)$ is a nearly representation of $\Pi^\lambda Q$.

Suppose first that $m'=1$. Then $(V,\theta)$ is necessarily an actual $\Pi^\lambda Q$-module by Lemma~\ref{lem:case-II}. It follows that $\Phi_{XX}(\xi_{XX})-\lambda 1_X$ is zero in position $VV$. As the $v$ component has rank at most one, and all other components are zero, it must also have a zero in either position $UV$ or position $VU$. Assume position $UV$ is zero. Then $\Ext^1_{KQ}(U,V)=0$, so $\Phi_{UV}$ is injective and $\xi_{UV}=0$. Thus $(V,\theta)$ is a $K\bar Q$-submodule of $(X,\xi)$. The argument when position $VU$ is zero is entirely analogous, in which case $U$ is a $K\bar Q$-submodule of $(X,\xi)$.

Suppose instead that $m'>1$. Then by induction, since $m'<m$, we know that the nearly representation $(V,\theta)$ has a proper non-trivial submodule $(W,\theta')$ such that $W$ is again regular. Let $Z$ be an $I$-graded vector space complement to $W$ in $V$. Then there are corresponding block decompositions of $X$ and $\xi$
\[
X_a = \begin{pmatrix} U_a&0&0\\0&W_a&\ast\\0&0&Z_a\end{pmatrix} \quad\textrm{and}\quad
\xi = \begin{pmatrix}\xi_{UU}&\xi_{UW}&\xi_{UZ} \\ \xi_{WU}&\xi_{WW}&\xi_{WZ} \\ \xi_{ZU}&0&\xi_{ZZ}\end{pmatrix}.
\]
Note that $\xi_{WW}=\theta'$, that $\xi_{ZW}=0$ since $W$ is a $K\bar Q$ submodule of $(V,\theta)$, and that the matrices $Z_a$ equip $Z$ with its natural structure as a $Q$-representation isomorphic to $V/W$. Using the formula for $\Phi_{XX}$ we have the block decomposition
\[
\Phi_{XX}(\xi) - \lambda 1_X = \begin{pmatrix} * & \Phi_{UW}(\xi_{UW})& *  \\ * & * & * \\ \Phi_{ZU}(\xi_{ZU}) & 0 & *  \end{pmatrix}. 
\]
As the $v$ component of $\Phi_{XX}(\xi)-\lambda 1_X$ has rank at most one and the other components are zero, we deduce that one of the blocks $UW$ or $ZU$ must also be zero.

Assume $\Phi_{UW}(\xi_{UW})=0$. Since $W$ is regular and a submodule of $V$, we have $\Ext^1_{KQ}(U,W)=0$. Thus $\Phi_{UW}$ is injective, so $\xi_{UW}=0$, and hence $W$ yields a $K\bar Q$-submodule of $(X,\xi)$.

Assume instead that $\Phi_{ZU}(\xi_{ZU})=0$. Then $\Ext^1_{KQ}(Z,U)=0$, so $\xi_{ZU}=0$, and hence $U\oplus W$ yields a $K\bar Q$-submodule of $(X,\xi)$.

\medskip
\noindent
\textit{Case 2 (Single tube, $Q$ without oriented cycles).}
Suppose that the indecomposable summands of $X$ lie in a single tube $\mathcal T$, and that $Q$ has no oriented cycles. Choose a $KQ$-module $T$ as in Lemma~\ref{l:perpT}. By Lemma~\ref{l:nearlyperp}, $(X,\xi)$ corresponds to a nearly representation $(Y,\eta)$ of $\Pi^{\lambda'} Q'$. Now the number of vertices of $Q'$ is strictly smaller than that of $Q$, and $\dimv Y = m\delta'$ and $\lambda'\cdot\delta'=0$ by Lemma~\ref{l:mijprop}.  Thus by induction $(Y,\eta)$ has a proper non-trivial subrepresentation. Since the functor $\iota$ of \S\ref{s:univloc} is exact, this gives a proper non-trivial subrepresentation $(X',\xi')$ of $(X,\xi)$. Moreover $X'\in T^\perp$, so $X'$ is regular.

\medskip
\noindent
\textit{Case 3 ($Q$ an oriented cycle).}
Suppose that $Q$ is an oriented cycle. If it has only one vertex, then there is a submodule by Lemma~\ref{l:loopcase}, so we may suppose that $Q$ has at least two vertices. The deformed preprojective algebra is independent of the orientation of the quiver \cite[Lemma 2.2]{CBH}, and by Lemma~\ref{l:nearlyinfty} the analogous result holds for nearly representations. Thus we can choose an orientation $Q'$ of $Q$ without oriented cycles. We now have a nearly $\Pi^\lambda Q'$-module $(Y,\eta)$ of dimension vector $m\delta$. Using Lemma~\ref{l:nonreg} or one of the cases above, we see that $(Y,\eta)$ is not simple. Thus $(X,\xi)$ is not simple.
\end{proof}

\section{Connection with previous work}
\label{s:theotherbit}
In \cite[Theorem 1.2]{CBmm} it is shown that the possible dimension vectors of simple modules for a deformed preprojective algebra $\Pi^\lambda Q$ are exactly the elements of a certain combinatorially defined set $\Sigma_\lambda$. A key part of the proof involves showing that there is no simple module in three cases  (I), (II) and (III) which arise in \cite[Theorem 8.1]{CBmm}. These three cases are dealt with on page 288 of \cite{CBmm}. Our Theorem~\ref{t:main} handles case (III), thereby avoiding the use of \cite[Theorem 9.1]{CBmm}. Case (II) is elementary, and is covered by our Lemma~\ref{lem:case-II}. Concerning case (I), our result handles the situation when $\lambda\cdot\delta=0$.

For completeness we offer the following proof of the remaining situation of case (I). Suppose $(X,\xi)$ is a $\Pi^\lambda Q$-module of dimension vector $mp\delta$, where $m\geq 2$, the field $K$ has characteristic $p>0$ and $\lambda\cdot\delta\neq 0$. We show that $X$ is not simple. Suppose first that $X=U\oplus V$ with $\Ext^1_{KQ}(U,V)=0$. Writing $\xi$ as a block matrix, we deduce that $\Phi_{UV}(\xi_{UV})=0$, and hence that $\xi_{UV}=0$. Thus $V$ yields a submodule of $(X,\xi)$ so $(X,\xi)$ is not simple. We reduce to the case when the indecomposable summands of $X$ all lie in a single tube $\mathcal T$. We can then apply the reduction process as above to reduce to the case when $Q$ is a loop. In this case the deformed preprojective algebra is $K\langle x,y\rangle/(xy-yx-\lambda)$ with $\lambda\neq0$. Rescaling $x$ or $y$, we may assume that $\lambda=1$, and hence we are dealing with the first Weyl algebra $A_1$.

The assertion now follows from the well-known fact that every simple $A_1$-module has dimension $p$, cf.~\cite{Revoy}. We sketch a proof. The centre of $A_1$ is $Z:=K[x^p,y^p]$, so given a simple $A_1$-module $M$, both $x^p$ and $y^p$ must act as scalars and then $M$ becomes a module for $A_1\otimes_ZK$. This latter algebra is central simple, so isomorphic to $\mathbb M_p(K)$ and we are done.

\end{document}